\documentclass[reqno,a4paper]{amsart}
\usepackage{amssymb}
\usepackage{amsmath}
\begin{document} 
\newtheorem{prop}{Proposition}[section]
\newtheorem{Def}{Definition}[section] \newtheorem{theorem}{Theorem}[section]
\newtheorem{lemma}{Lemma}[section] \newtheorem{Cor}{Corollary}[section]

\title[Chern-Simons-Higgs in temporal gauge]{\bf Low regularity local well-posedness for the Chern-Simons-Higgs system in temporal gauge}
\author[Hartmut Pecher]{
{\bf Hartmut Pecher}\\
Fachbereich Mathematik und Naturwissenschaften\\
Bergische Universit\"at Wuppertal\\
Gau{\ss}str.  20\\
42097 Wuppertal\\
Germany\\
e-mail {\tt pecher@math.uni-wuppertal.de}}
\date{}

\begin{abstract}
The Cauchy problem for the Chern-Simons-Higgs system in the (2+1)-dimensional Minkowski space in temporal gauge is locally well-posed for low regularity initial data improving a result of Huh. The proof uses the bilinear space-time estimates in wave-Sobolev spaces by d'Ancona, Foschi and Selberg and takes advantage of a null condition.
\end{abstract}
\maketitle
\renewcommand{\thefootnote}{\fnsymbol{footnote}}
\footnotetext{\hspace{-1.5em}{\it 2000 Mathematics Subject Classification:} 
35Q40, 35L70 \\
{\it Key words and phrases:} Chern-Simons-Higgs,  
local well-posedness, temporal gauge}
\normalsize 
\setcounter{section}{0}
\section{Introduction and main results}
\noindent Consider the Chern-Simons-Higgs system in the Minkowski space 
${\mathbb R}^{1+2} = {\mathbb R}_t \times {\mathbb R}_x^2$ with metric $g_{\mu \nu} = diag(1,-1,-1)$ :
\begin{align}
\label{*1}
 F_{\mu \nu} & =  2 \epsilon^{\mu \nu \rho} Im(\overline{\phi} D^{\rho} \phi) \\
\label{*2}
D_{\mu} D^{\mu} \phi & = - \phi V'(|\phi|^2) \, ,
\end{align}
with initial data
\begin{equation}
\label{*3}
A_{\nu}(0) = a_{\nu} \, , \, \phi(0) = \phi_0 \, , \, (\partial_0 \phi)(0) = \phi_1 \, , 
\end{equation}
where we use the convention that repeated upper and lower indices are summed, Greek indices run over 0,1,2 and Latin indices over 1,2. Here 
\begin{align*}
D^{\mu}  & := \partial_{\mu} - iA_{\mu} \\
 F_{\mu \nu} & := \partial_{\mu} A_{\nu} - \partial_{\nu} A_{\mu} \\
\end{align*}
Here $F_{\mu \nu} : {\mathbb R}^{1+2} \to {\mathbb R}$ denotes the curvature, $\phi : {\mathbb R}^{1+2} \to {\mathbb C}$ a scalar field and $A_{\nu} : {\mathbb R}^{1+2} \to {\mathbb R}$ the gauge potentials. We use the notation $\partial_{\mu} = \frac{\partial}{\partial x_{\mu}}$, where we write $(x^0,x^1,...,x^n) = (t,x^1,...,x^n)$ and also $\partial_0 = \partial_t$ and $\nabla = (\partial_1,\partial_2)$. $\epsilon^{\mu \nu \rho}$ is the totally skew-symmetric tensor with $\epsilon^{012} = 1$, and the Higgs potential $V$ is assumed to fulfill
$V \in C^{\infty}({\mathbb R}_+,{\mathbb R}) $ , $V(0)=0$ and all derivatives of $V$ have polynomial growth.

This model was proposed by Hong, Kim and Pac \cite{HKP} and Jackiw and Weinberg \cite{JW} in the study of vortex solutions in the abelian Chern-Simons theory.

The equations are invariant under the gauge transformations
$$ A_{\mu} \rightarrow A'_{\mu} = A_{\mu} + \partial_{\mu} \chi \, , \, \phi \rightarrow \phi' = e^{i\chi} \phi \, , \, D_{\mu} \rightarrow D'_{\mu} = \partial_{\mu}-iA'_{\mu} \, . $$
The most common gauges are the Coulomb gauge $\partial^j A_j =0$ , the Lorenz gauge $\partial^{\mu} A_{\mu} = 0$ and the temporal gauge $A_0 = 0$. In this paper we exclusively study the temporal gauge for low regularity data.

Global well-posedness in the Coulomb gauge was proven by Chae and Choe \cite{CC} for data $a_{\mu} \in H^a$ , $\phi_0 \in H^b$ , $\phi_1 \in H^{b-1}$ where $(a,b)=(l,l+1)$ with $l \ge 1$ , satisfying a compatibility condition and a class of Higgs potentials. Huh \cite{H} showed local well-posedness in the Coulomb gauge for $(a,b)=(\epsilon,1+\epsilon)$ , in the Lorenz gauge for $(a,b) = (\frac{3}{4}+\epsilon,\frac{9}{8}+\epsilon)$ or $(a,b)=(\frac{1}{2},\frac{3}{2})$ , and in the temporal gauge for $(a,b)=(l,l)$ with $l \ge \frac{3}{2}$ . He also showed global well-posedness in the temporal gauge for $l \ge 2$ . The local well-posedness result in the Lorenz gauge was improved to $(a,b)=(l,l+1)$ and $l > \frac{1}{4}$ by Bournaveas \cite{B} and by Yuan \cite{Y}.  Also in Lorenz gauge the very important global well-posedness result in energy space, where $a_{\mu} \in \dot{H}^{\frac{1}{2}}$ , $\phi_0 \in H^1$ , $\phi_1 \in L^2$ , was proven by Selberg and Tesfahun \cite{ST} under suitable assumptions on the potential V, and even unconditional well-posedness could be proven by Selberg and Oliveira da Silva \cite{SO}. In \cite{ST} the regularity assumptions on the data could also be lowered down in Lorenz gauge to $(a,b)=(l,l+1)$ and $l > \frac{3}{8}$ . This latter result was improved to $l > \frac{1}{4}$ by Huh and Oh \cite{HO}. Global well-posedness in energy space and local well-posedness for $a_{\mu} \in \dot{H}^{\frac{1}{2}}$ , $\phi_0 \in H^{l+\frac{1}{2}}$ , $\phi_1 \in H^{l-\frac{1}{2}}$ for $\frac{1}{2} \ge l > \frac{1}{4}$ in Coulomb gauge was very recently obtained by Oh \cite{O}. For all these results up to the paper by Chae and Choe \cite{CC} and Oh \cite{O} it was crucial to make use  of a null condition in the nonlinearity of the system.

An earlier result in the temporal gauge was also given by Tao \cite{T} for the Yang-Mills equations.

In this paper we consider exclusively the temporal gauge and obtain local well-posedness for data $a_j \in H^{s+\frac{3}{4}}$ , $\phi_0 \in H^{s+1}$ , $\phi_1 \in H^{s}$ under the conditon $s > \frac{1}{4}$ and the compatibilty assumption $\partial_1 a_2 - \partial_2 a_1 = 2 Im(\overline{\phi}_0 \phi_1)$.  We use a contraction argument in $X^{s,b}$ - type spaces adapted to the phase functions $\tau \pm |\xi|$ on one hand and to the phase function $\tau$ on the other hand. We also take advantage of a null condition which appears in the nonlinearity. Most of the crucial arguments follow from the bilinear estimates in wave-Sobolev spaces established by d'Ancona, Foschi and Selberg \cite{AFS}.

We denote the Fourier transform with respect to space and time by $\,\widehat{}$ . The operator
$|\nabla|^{\alpha}$ is defined by $(|\nabla|^{\alpha} f)(\xi) = |\xi|^{\alpha} ({\mathcal F}f)(\xi)$ and similarly $ \langle \nabla \rangle^{\alpha}$. The inhomogeneous and homogeneous Sobolev spaces are denoted by $H^{s,p}$ and $\dot{H}^{s,p}$, respectively. For $p=2$ we simply denote them by $H^s$ and $\dot{H}^s$. We repeatedly use the Sobolev embeddings $\dot{H}^{s,p} \subset L^q$ for  $1<p\le q < \infty$ and $\frac{1}{q} = \frac{1}{p}-\frac{s}{2}$, and also $\dot{H}^{1+} \cap \dot{H}^{1-} \subset L^{\infty}$ in two space dimensions. \\
$a+ := a + \epsilon$ for a sufficiently small $\epsilon >0$ , so that $a<a+<a++$ , and similarly $a--<a-<a$ , and $\langle \cdot \rangle := (1+|\cdot|^2)^{\frac{1}{2}}$ .

We now formulate our main result and begin by defining the standard spaces $X^{s,b}_{\pm}$ of Bourgain-Klainerman-Machedon type belonging to the half waves as the completion of the Schwarz space  $\mathcal{S}({\mathbb R}^3)$ with respect to the norm
$$ \|u\|_{X^{s,b}_{\pm}} = \| \langle \xi \rangle^s \langle  \tau \pm |\xi| \rangle^b \widehat{u}(\tau,\xi) \|_{L^2_{\tau \xi}} \, . $$ 
Similarly we define the wave-Sobolev spaces $X^{s,b}_{|\tau|=|\xi|}$ with norm
$$ \|u\|_{X^{s,b}_{|\tau|=|\xi|}} = \| \langle \xi \rangle^s \langle  |\tau| - |\xi| \rangle^b \widehat{u}(\tau,\xi) \|_{L^2_{\tau \xi}}  $$ and also $X^{s,b}_{\tau =0}$ with norm 
$$\|u\|_{X^{s,b}_{\tau=0}} = \| \langle \xi \rangle^s \langle  \tau  \rangle^b \widehat{u}(\tau,\xi) \|_{L^2_{\tau \xi}} \, .$$
We also define $X^{s,b}_{\pm}[0,T]$ as the space of the restrictions of functions in $X^{s,b}_{\pm}$ to $[0,T] \times \mathbb{R}^2$ and similarly $X^{s,b}_{|\tau| = |\xi|}[0,T]$ and $X^{s,b}_{\tau =0}[0,T]$. We frequently use the estimates $\|u\|_{X^{s,b}_{\pm}} \le \|u\|_{X^{s,b}_{|\tau|=|\xi|}}$ for $b \le 0$ and the reverse estimate for $b \ge 0$. 

Our main theorem reads as follows:                                  
\begin{theorem}
\label{Theorem}
Assume $ s > \frac{1}{4}$ and $\epsilon > 0$ sufficiently small. The Chern-Simons-Higgs system (\ref{*1}),(\ref{*2}),(\ref{*3}) in temporal gauge $A_0=0$ with data $\phi_0 \in H^{s+1}({\mathbb R}^2)$ , $ \phi_1 \in H^{s}({\mathbb R}^2),$ $|\nabla|^{\epsilon} a_j \in H^{s+\frac{3}{4}-\epsilon}({\mathbb R}^2) $ satisfying the compatibility condition $\partial_1 a_2 - \partial_2 a_1 = 2 Im(\overline{\phi}_0 \phi_1)$ has a unique local solution $\phi \in C^0([0,T],H^{s+1}({\mathbb R}^2)) \cap C^1([0,T],H^{s}({\mathbb R}^2)),$  $|\nabla|^{\epsilon} A \in C^0([0,T],H^{s+\frac{3}{4}-\epsilon}({\mathbb R}^2))$.  \\
More precisely we have $\phi \in X^{s+1,\frac{1}{2}+}_+[0,T] + X^{s+1,\frac{1}{2}+}_-[0,T] ,$ $A= A^{cf} + A^{df}$ , where $ \nabla A^{cf} \in X^{{s-\frac{1}{4}},\frac{1}{2}+}_{\tau=0} [0,T]$ , $|\nabla|^{\epsilon} A^{cf} \in C^0([0,T],L^2({\mathbb R}^2))$ , $\nabla A^{df} \in X^{{s},\frac{1}{2}+}_{|\tau|=|\xi|}[0,T],$ $|\nabla|^{\epsilon} A^{df} \in C^0([0,T],L^2({\mathbb R}^2))$ and in these spaces uniqueness holds, and moreover  $\nabla A^{cf} \in X^{s-\frac{1}{4},1-}_{\tau =0}[0,T]$. 

\end{theorem}

\section{Reformulation of the problem}
In the temporal gauge $A_0 =0$ the Chern-Simons-Higgs system (\ref{*1}),(\ref{*2}) is equivalent to the following system
\begin{align}
\label{**1}
&\partial_t A_j   = 2 \epsilon_{ij} Im(\overline{\phi} D^i \phi) \\
\label{**2}
&\partial_t^2 \phi - D^j D_j \phi  = - \phi V'(|\phi|^2) \\
\nonumber
&\Longleftrightarrow \Box \phi  = 2i A^j \partial_j \phi - i \partial_j A^j \phi + A^j A_j \phi - \phi V'(|\phi|^2) \\
\label{1.9}
&\partial_1 A_2 - \partial_2 A_1  = 2 Im(\overline{\phi} \partial_t \phi) \, ,
\end{align}
where $i,j=1,2$ , $\epsilon_{12}=1$ , $\epsilon_{21}=-1$ and $\Box = \partial_t^2 - \partial_1^2 - \partial_2^2$ . \\
We remark that (\ref{1.9}) is fulfilled for any solution of (\ref{**1}),(\ref{**2}), if it holds initially, i.e. , if the following compatibility condition holds, which we assume from now on:
\begin{equation}
\label{1.9init}
\partial_1 A_2(0) - \partial_2 A_1(0) = 2 Im(\overline{\phi}(0) (\partial_t \phi)(0)) \, .
\end{equation}
Indeed, we have by (\ref{**1}) and (\ref{**2}):
$$ \partial_t(\partial_1 A_2 - \partial_2 A_1) = 2 Im(\overline{\phi}(D_1^2 \phi + D_2^2 \phi)) = 2 Im(\overline{\phi} \partial_t^2 \phi) = 2 \partial_t Im(\overline{\phi} \partial_t \phi) \, . $$
Thus we only have to solve (\ref{**1}) and (\ref{**2}), and can assume that (\ref{1.9}) is fulfilled. We make the standard decomposition of $A=(A_1,A_2)$ into its divergence-free part $A^{df}$ and its curl-free part $A^{cf}$, namely $A=A^{df} +A^{cf}$, where
\begin{align*}
A^{df} & = (-\Delta)^{-1}(\partial_1 \partial_2 A_2 - \partial_2^2 A_1,\partial_1 \partial_2 A_1 - \partial_1^2 A_2) \, , \\
A^{cf} & = (\Delta)^{-1}(\partial_1 \partial_2 A_2 + \partial_1^2 A_1, \partial_1 \partial_2 A_1 + \partial_2^2 A_2) \, .
\end{align*}
Let $B$ be defined by $A^{df}_1 = -\partial_2 B$ , $A^{df}_2 = \partial_1 B$ . Then by (\ref{1.9}) and $\partial_1 A^{cf}_2 - \partial_2 A^{cf}_1 =0$ we obtain
$$ \Delta B = \partial_1 A_2^{df} - \partial_2 A_1^{df} = 2 Im(\overline{\phi} \partial_t \phi) \, ,$$
so that
\begin{equation}
\label{***2}
A^{df}_1 = -2 \Delta^{-1} \partial_2 Im(\overline{\phi} \partial_t \phi) \, , \, A^{df}_2 = 2 \Delta^{-1} \partial_1 Im(\overline{\phi} \partial_t \phi) \, .
\end{equation}
Next we calculate $\partial_t A^{cf}$ for solutions $(A,\phi)$ of (\ref{**1}),(\ref{**2}):
\begin{align}
\nonumber
\partial_t A^{cf}_1 &= \Delta^{-1} \partial_1(\partial_2 \partial_t A_2 + \partial_1 \partial_t A_1) \\
\nonumber
& = 2 \Delta^{-1} \partial_1(\partial_2 Im(\overline{\phi} D_1 \phi) - \partial_1 Im(\overline{\phi} D_2 \phi) \\
\nonumber
& = 2 \Delta^{-1} \partial_1 Im(\partial_2 \overline{\phi} D_1 \phi + \overline{\phi} \partial_2 D_1 \phi - \partial_1 \overline{\phi} D_2 \phi - \overline{\phi} \partial_1 D_2 \phi) \\
\nonumber
& = 2 \Delta^{-1} \partial_1 Im [\partial_2 \overline{\phi}\partial_1 \phi - \partial_1 \overline{\phi} \partial_2 \phi -i A_1 \partial_2 \overline{\phi} \phi + i A_2 \partial_1 \overline{\phi} \phi \\
\nonumber
& \quad + \overline{\phi}(-i A_1 \partial_2 \phi + i A_2 \partial_1 \phi -i \partial_2 A_1 \phi + i \partial_1 A_2 \phi) ] \\
\nonumber
& = 2 \Delta^{-1} \partial_1 Im [\partial_2 \overline{\phi} \partial_1 \phi - \partial_1 \overline{\phi} \partial_2 \phi + i \overline{\phi}(A_2 \partial_1 \phi -A_1 \partial_2 \phi) \\
\nonumber
&  \quad + i \phi(A_2 \partial_1 \overline{\phi} - A_1 \partial_2 \overline{\phi}) + i(\partial_1 A_2 - \partial_2 A_1) |\phi|^2 ] \\
\nonumber
& = 2 \Delta^{-1} \partial_1 Im(\partial_2 \overline{\phi} \partial_1 \phi - \partial_1 \overline{\phi} \partial_2 \phi)  + 2 \Delta^{-1} \partial_1(A_2 \partial_1 |\phi|^2 -A_1 \partial_2 |\phi|^2) \\
\label{***3a}
& \quad + 4 \Delta^{-1} \partial_1 Im(\overline{\phi} \partial_t \phi) |\phi|^2 \, .
\end{align}
Similarly
\begin{align}
\nonumber
 \partial_t A^{cf}_2 & = 2 \Delta^{-1} \partial_2 Im(\partial_1 \overline{\phi} \partial_2 \phi - \partial_2 \overline{\phi} \partial_1 \phi)  + 2 \Delta^{-1} \partial_2(A_1 \partial_2 |\phi|^2 -A_2 \partial_1 |\phi|^2) \\
\label{***3b}
& \quad + 4 \Delta^{-1} \partial_2 Im(\overline{\phi} \partial_t \phi) |\phi|^2 \, .
\end{align}

Moreover from (\ref{**2}) we obtain using $\partial^j A_j^{df} =0$:
\begin{equation}
\label{***1}
\Box \phi = 2i A^{cf} \nabla \phi +2i A^{df} \nabla \phi -i \partial^j A_j^{cf} \phi + (A^{df,j} + A^{cf,j})(A^{df}_j + A^{cf}_j) \phi - \phi V'(|\phi|^2)  .
\end{equation}

We also obtain from (\ref{***2}) and (\ref{**2})
$$ \partial_t A^{df}_2 = 2 \partial_t \Delta^{-1} \partial_1 Im(\overline{\phi} \partial_t \phi) =2 \Delta^{-1} \partial_1 Im(\overline{\phi} \partial_t^2 \phi) = 2\Delta^{-1} \partial_1 Im(\overline{\phi} D^j D_j \phi) \, . $$
Now
\begin{align*}
Im(\overline{\phi} D^j D_j \phi) & = Im(\overline{\phi} \partial^j \partial_j \phi - 2i \overline{\phi} A_j \partial^j \phi - i \overline{\phi} \partial^j A_j \phi) \\
& = Im(\overline{\phi} \partial^j \partial_j \phi - i \overline{\phi} A_j \partial^j \phi -i A_j \phi \partial^j \overline{\phi}- i \overline{\phi} \partial^j A_j \phi) \\
& = \partial^j Im(\overline{\phi} D_j \phi) \, ,
\end{align*}
so that
\begin{equation}
\label{****1}
\partial_t A^{df}_2 = 2\Delta^{-1} \partial_1 \partial^j Im(\overline{\phi} D_j \phi) \, .
\end{equation}
Similarly we obtain
\begin{equation}
\label{****2}
\partial_t A^{df}_1 = - 2\Delta^{-1} \partial_2 \partial^j Im(\overline{\phi} D_j \phi) \, .
\end{equation}
Reversely defining $A:= A^{df} + A^{cf}$ we show that our new system (\ref{***2}),(\ref{***3a}),(\ref{***3b}),(\ref{***1}) implies (\ref{**1}),(\ref{**2}) and also (\ref{1.9}), provided the compatability condition (\ref{1.9init}) is fulfilled. (\ref{**2}) is obvious. (\ref{1.9}) is fulfilled because by use of (\ref{***3a}) and (\ref{***3b})  one easily checks $\partial_1 A^{cf}_2 - \partial_2 A^{cf}_1 = 0$, so that by (\ref{***2})
$$ \partial_t (\partial_1 A_2 - \partial_2 A_1) = \partial_t( \partial_1 A^{df}_2 - \partial_2 A^{df}_1) = \partial_t Im(\overline{\phi} \partial_t \phi) \, . $$
Thus (\ref{1.9}) is fulfilled, if (\ref{1.9init}) holds. Finally we obtain
\begin{align*}
&\partial_t A_1  = \partial_t A_1^{cf} + \partial_t A^{df}_1 \\
& = 2 \Delta^{-1} \partial_1(\partial_2 Im(\overline{\phi} D_1 \phi) - \partial_1 Im(\overline{\phi} D_2 \phi))\hspace{-0.1em} - \hspace{-0.1em} 2 \Delta^{-1} \partial_2(\partial_1 Im(\overline{\phi} D_1 \phi) + \partial_2 Im(\overline{\phi} D_2 \phi)) \\
& = - 2 Im(\overline{\phi} D_2 \phi) \, ,
\end{align*}
where we used (\ref{***3a}) and also (\ref{****2}), which was shown to be a consequence of (\ref{***2}) and (\ref{***1}).
Similarly we also get
$$ \partial _t A_2 = 2 Im(\overline{\phi} D_1 \phi) \, , $$
so that (\ref{**1}) is shown to be satisfied.

Summarizing we have shown that (\ref{**1}),(\ref{**2}),(\ref{1.9}) are equivalent to (\ref{***2}),(\ref{***3a}),(\ref{***3b}),(\ref{***1}) (which also implies (\ref{****1}),(\ref{****2})).

Concerning the initial conditions assume we are given initial data for our system (\ref{**1}),(\ref{**2}),(\ref{1.9}):
$$ A_j(0) = a_j \, , \, \phi(0) = \phi_0 \, , \, (\partial_t \phi)(0) = \phi_1 $$
satisfying $|\nabla|^{\epsilon}a_j \in H^{s+\frac{3}{4}-\epsilon}$ , $\phi_0 \in H^{s+1} $ , $ \phi_1 \in H^{s}$ and (\ref{1.9init}). Then by (\ref{***2}) and (\ref{1.9init}) we obtain
\begin{align*}
A^{df}_1(0) &= -2 \Delta^{-1} \partial_2 Im(\overline{\phi}_0 \phi_1) = - \Delta^{-1} \partial_2(\partial_1 a_2 - \partial_2 a_1)  \\
A^{df}_2(0)& = 2 \Delta^{-1} \partial_1 Im(\overline{\phi}_0 \phi_1) =  \Delta^{-1} \partial_1(\partial_1 a_2 - \partial_2 a_1) 
\end{align*}
and 
$$ A^{cf}_j(0) = a_j - A^{df}_j(0)  ,$$
thus $$|\nabla|^{\epsilon} A^{df}_j(0)\in H^{s+\frac{3}{4}-\epsilon} \, , \, |\nabla|^{\epsilon}A^{cf}_j(0) \in H^{s+\frac{3}{4}-\epsilon}\, .$$
In the sequel we construct a solution of the Cauchy problem for (\ref{***2}),(\ref{***3a}),(\ref{***3b}), (\ref{***1}) with data $\phi_0 \in H^{s+1}$ , $ \phi_1 \in H^{s}$ , $|\nabla|^{\epsilon} A^{cf}_j(0) \in H^{s+\frac{3}{4}-\epsilon}$. We have shown that whenever we have a local solution of this system with data $\phi_0,\phi_1$ and $A_j^{cf}(0) = a_j - A^{df}_j (0)$, where $A^{df}_1(0) = -2 \Delta^{-1} \partial_2 Im(\overline{\phi}_0 \phi_1)$ , 
$A^{df}_2(0) = 2 \Delta^{-1} \partial_1 Im(\overline{\phi}_0 \phi_1)$, we also have that $(\phi,A)$ with $A:= A^{df} + A^{cf}$ is a local solution of (\ref{**1}),(\ref{**2}) with data $(\phi_0,\phi_1,a_1,a_2)$. If (\ref{1.9init}) holds then (\ref{1.9}) is also satisfied.

Defining
$$ \phi_{\pm} = \frac{1}{2}(\phi \pm i^{-1} \langle \nabla \rangle^{-1} \partial_t \phi) \, \Longleftrightarrow \, \phi=\phi_+ +\phi_- \, , \, \partial_t \phi= i \langle \nabla \rangle (\phi_+ - \phi_-) $$
the equation (\ref{***1}) transforms to
\begin{align}
\label{***1'}
(i \partial_t \pm \langle \nabla \rangle) \phi_{\pm} 
 = &\pm 2^{-1} \langle \nabla \rangle^{-1} \big(2i A^{cf} \nabla \phi +2i A^{df} \nabla \phi -i \partial^j A_j^{cf} \phi \\
 \nonumber
 &+ (A^{df,j} + A^{cf,j})(A^{df}_j + A^{cf}_j) \phi - \phi V'(|\phi|^2) + \phi \big)
\end{align}

Fundamental for the proof of our theorem are the following bilinear estimates in wave-Sobolev spaces which were proven by d'Ancona, Foschi and Selberg in the two-dimensional case $n=2$ in \cite{AFS} in a more general form which include many limit cases which we do not need.
\begin{theorem}
\label{Theorem3}
Let $n=2$. The estimate
$$\|uv\|_{X_{|\tau|=|\xi|}^{-s_0,-b_0}} \lesssim \|u\|_{X^{s_1,b_1}_{|\tau|=|\xi|}} \|v\|_{X^{s_2,b_2}_{|\tau|=|\xi|}} $$ 
holds, provided the following conditions hold:
\begin{align*}
\nonumber
& b_0 + b_1 + b_2 > \frac{1}{2} \\
\nonumber
& b_0 + b_1 > 0 \\
\nonumber
& b_0 + b_2 > 0 \\
\nonumber
& b_1 + b_2 > 0 \\
\nonumber
&s_0+s_1+s_2 > \frac{3}{2} -(b_0+b_1+b_2) \\
\nonumber
&s_0+s_1+s_2 > 1 -\min(b_0+b_1,b_0+b_2,b_1+b_2) \\
\nonumber
&s_0+s_1+s_2 > \frac{1}{2} - \min(b_0,b_1,b_2) \\
\nonumber
&s_0+s_1+s_2 > \frac{3}{4} \\
\end{align*}
\begin{align*}
 &(s_0 + b_0) +2s_1 + 2s_2 > 1 \\
\nonumber
&2s_0+(s_1+b_1)+2s_2 > 1 \\
\nonumber
&2s_0+2s_1+(s_2+b_2) > 1 \\
\nonumber
&s_1 + s_2 \ge \max(0,-b_0) \\
\nonumber
&s_0 + s_2 \ge \max(0,-b_1) \\
\nonumber
&s_0 + s_1 \ge \max(0,-b_2)   \, .
\end{align*}
\end{theorem}

\section{Proof of Theorem \ref{Theorem}}
Taking the considerations of the previous section into account Theorem \ref{Theorem} reduces to the following proposition and its corollary.
\begin{prop}
Assume $s > \frac{1}{4}$.
The system 
\begin{align}
\label{***1''}
(i \partial_t \pm \langle \nabla \rangle) \phi_{\pm} &
 = \pm 2^{-1} \langle \nabla \rangle^{-1} \big(2i A^{cf} \nabla \phi +2i A^{df} \nabla \phi -i\, \partial^j A_j^{cf} \phi \\
 \nonumber
 & \quad + (A^{df,j} + A^{cf,j})(A^{df}_j + A^{cf}_j) \phi - \phi V'(|\phi|^2) + \phi \big) \\
\label{***2'}
A^{df}_1 &= -2 \Delta^{-1} \partial_2 Im(\overline{\phi} \partial_t \phi) \quad , \quad
 A^{df}_2 = 2 \Delta^{-1} \partial_1 Im(\overline{\phi} \partial_t \phi) \\ 
 \nonumber
\partial_t A^{cf}_1 &= 2 \Delta^{-1} \partial_1 Im(\partial_2 \overline{\phi} \partial_1 \phi - \partial_1 \overline{\phi} \partial_2 \phi)  + 2 \Delta^{-1} \partial_1(A_2 \partial_1 |\phi|^2 -A_1 \partial_2 |\phi|^2) \\
\label{***3a'} 
& \quad + 4 \Delta^{-1} \partial_1 Im(\overline{\phi} \partial_t \phi) |\phi|^2 \, \\
\nonumber
 \partial_t A^{cf}_2 & = 2 \Delta^{-1} \partial_2 Im(\partial_1 \overline{\phi} \partial_2 \phi - \partial_2 \overline{\phi} \partial_1 \phi)  + 2 \Delta^{-1} \partial_2(A_1 \partial_2 |\phi|^2 -A_2 \partial_1 |\phi|^2) \\
\label{***3b'}
& \quad + 4 \Delta^{-1} \partial_2 Im(\overline{\phi} \partial_t \phi) |\phi|^2 \, . 
\end{align}
with data
$ \phi_{\pm}(0) \in H^{s+1} $ and $|\nabla|^{\epsilon}A^{cf}(0) \in H^{s+\frac{3}{4}-\epsilon}$ has a unique local solution $$\phi_{\pm} \in X^{s+1,\frac{1}{2}+}_{\pm}[0,T] \, , \, \nabla A^{cf} \in X^{s-\frac{1}{4},\frac{1}{2}+}_{\tau =0}[0,T] \, , \, |\nabla|^{\epsilon} A^{cf} \in C^0([0,T],L^2) \, .$$ 
Here $\phi=\phi_+ +\phi_- \, , \, \partial_t \phi= i \langle \nabla \rangle (\phi_+ - \phi_-)$ .
Moreover $A^{df}$ satisfies $\nabla A^{df} \in X^{s,\frac{1}{2}+}_{|\tau|=|\xi|}[0,T]$ , $|\nabla|^{\epsilon}A^{df} \in C^0([0,T],L^2)$ and  also  $\nabla A^{cf} \in X^{s-\frac{1}{4},1-}_{\tau =0}[0,T]$ .
\end{prop}
We obtain immediately
\begin{Cor}
The solution has the property
 $ \, \phi \in C^0([0,T],H^{s+1}) \cap C^1([0,T],H^{s}),$  $|\nabla|^{\epsilon} A^{cf} \in C^0([0,T],H^{s+\frac{3}{4}-\epsilon})$ and $|\nabla|^{\epsilon} A^{df} \in C^0([0,T],H^{s+1-\epsilon}) \, . $
\end{Cor}
\begin{proof}
We want to apply the contraction mapping principle for $$\phi_{\pm} \in X^{s+1,\frac{1}{2}+}_{\pm}[0,T] \, , \, \nabla A^{cf} \in X^{s-\frac{1}{4},\frac{1}{2}+}_{\tau =0}[0,T] \, , \, |\nabla|^{\epsilon} A^{cf} \in C^0([0,T],L^2) \, .$$ 
By well-known arguments this is reduced to the estimates of the right hand sides of (\ref{***1''}),(\ref{***3a'}) and (\ref{***3b'}) stated as claims 1-9 below. We start to control $\nabla A^{cf}$ in $ X^{s-\frac{1}{4},\frac{1}{2}+}_{\tau =0}$.\\
{\bf Claim 1:} $$ \|\partial_i \overline{\phi} \partial_j \phi - \partial_j \overline{\phi} \partial_i \phi \|_{X^{s-\frac{1}{4},0}_{\tau =0}} \lesssim \|\nabla \phi\|^2_{X^{s,\frac{1}{2}+}_{|\tau| = |\xi|}} \, . $$
Using ($\pm_1$ and $\pm_2$ denote independent signs)
$$\partial_i \overline{\phi} \partial_j \phi - \partial_j \overline{\phi} \partial_i \phi  = \sum_{\pm_1,\pm_2} (\partial_i \overline{\phi}_{\pm_1} \partial_j \phi_{\pm_2} - \partial_j \overline{\phi}_{\pm_1} \partial_i \phi_{\pm_2})$$ 
it suffices to show 
$$ \|\partial_i \overline{\phi} \partial_j \psi - \partial_j \overline{\phi} \partial_i \psi \|_{X^{s-\frac{1}{4},0}_{\tau =0}} \lesssim \|\nabla \phi\|_{X^{s,\frac{1}{2}+}_{\pm_1}} \|\nabla \psi\|_{X^{s,\frac{1}{2}+}_{\pm_2}} \, . $$
We now use the null structure of this term in the form that for vectors $\xi=(\xi^1,\xi^2),$  $\eta=(\eta^1,\eta^2) \in {\mathbb R}^2$ the following estimate holds
$$ |\xi^i \eta^j - \xi^j \eta^i| \le |\xi| |\eta| \angle(\xi,\eta) \, , $$
where $\angle(\xi,\eta)$ denotes the angle between $\xi$ and $\eta$, and the following lemma gives the decisive bound for the angle:
\begin{lemma} (\cite{S}, Lemma 2.1 or \cite{ST}, Lemma 3.2)
\begin{equation}
\label{angle}
\angle(\pm_1 \xi_1,\pm_2 \xi_2) \lesssim \Big(\frac{\langle \tau_1 \pm_1 |\xi_1| \rangle + \langle \tau_2 \pm_2 |\xi_2| \rangle + \langle |\tau_3|-|\xi_3| \rangle}{\min(\langle \xi_1 \rangle,\langle \xi_2 \rangle)}\Big)^{\frac{1}{2}}
\end{equation}
$\forall \, \xi_1,\xi_2,\xi_3 \in{\mathbb R}^2 \, , \, \tau_1,\tau_2,\tau_3 \in {\mathbb R}$ with $\xi_1+\xi_2+\xi_3=0$ and $\tau_1+\tau_2+\tau_3 =0$.
\end{lemma}
\noindent Thus the claimed estimate reduces to 
\begin{align}
\nonumber
\Big| \int_* \frac{\widehat{u}_1(\tau_1,\xi_1)}{\langle \tau_1 \pm_1 |\xi_1|\rangle^{\frac{1}{2}+} \langle \xi_1 \rangle^s} 
\frac{\widehat{u}_2(\tau_2,\xi_2)}{\langle \tau_2 \pm_2 |\xi_2|\rangle^{\frac{1}{2}+} \langle \xi_2 \rangle^s}
\frac{\widehat{u}_3(\tau_3,\xi_3)}{\langle \xi_3 \rangle^{\frac{1}{4}-s}} \angle(\pm_1 \xi_1,\pm_2 \xi_2) \Big| 
 \\
 \label{20}
\lesssim \|u_1\|_{L^2_{xt}}\|u_2\|_{L^2_{xt}}\|u_3\|_{L^2_{xt}} \, ,
\end{align}
where * denotes integration over  $\xi_1+\xi_2+\xi_3=0$ and $\tau_1+\tau_2+\tau_3 =0$. We assume without loss of generality $|\xi_1| \le |\xi_2|$ and the Fourier transforms are nonnnegative. We distinguish three cases according to which of the modules on the right hand side of (\ref{angle}) is dominant. \\
Case 1: $\langle|\tau_3| - |\xi_3|\rangle$ dominant. \\
In this case (\ref{20}) reduces to
\begin{align*}
\Big| \int_* \frac{\widehat{u}_1(\tau_1,\xi_1)}{\langle |\tau_1| - |\xi_1|\rangle^{\frac{1}{2}+} \langle \xi_1 \rangle^{s+\frac{1}{2}}} 
\frac{\widehat{u}_2(\tau_2,\xi_2)}{\langle |\tau_2| - |\xi_2|\rangle^{\frac{1}{2}+} \langle \xi_2 \rangle^s}
\frac{\widehat{u}_3(\tau_3,\xi_3)}{\langle \xi_3 \rangle^{\frac{1}{4}-s}} \langle |\tau_3|-|\xi_3 | \rangle^{\frac{1}{2}} \Big|
 \\
\lesssim \|u_1\|_{L^2_{xt}}\|u_2\|_{L^2_{xt}}\|u_3\|_{L^2_{xt}} \, .
\end{align*}
This follows from Theorem \ref{Theorem3} with $s_0 = \frac{1}{4}-s$ , $b_0 = -\frac{1}{2}$ , $s_1=s+\frac{1}{2}$ , $b_1=b_2=\frac{1}{2}+,$ $s_2 = s$. Its assumptions are satisfied, because $s_0+s_1+s_2=\frac{3}{4}+s > 1=\frac{1}{2}-b_0$ under our assumption $s > \frac{1}{4}$. \\
Case 2: $\langle \tau_1 \pm_1 |\xi_1| \rangle$ dominant. \\
This case is reduced to
\begin{align*}
\Big| \int_* \frac{\widehat{u}_1(\tau_1,\xi_1)}{\langle |\tau_1| - |\xi_1|\rangle^{0+} \langle \xi_1 \rangle^{s+\frac{1}{2}}} 
\frac{\widehat{u}_2(\tau_2,\xi_2)}{\langle |\tau_2| - |\xi_2|\rangle^{\frac{1}{2}+} \langle \xi_2 \rangle^s}
\frac{\widehat{u}_3(\tau_3,\xi_3)}{\langle \xi_3 \rangle^{\frac{1}{4}-s}} \Big|
 \\
\lesssim \|u_1\|_{L^2_{xt}}\|u_2\|_{L^2_{xt}}\|u_3\|_{L^2_{xt}} \, ,
\end{align*}
which follows from Theorem \ref{Theorem3} with $s_0 = \frac{1}{4}-s$ , $b_0 = 0$ , $s_1=s+\frac{1}{2}$ , $b_1= 0+,$  $b_2=\frac{1}{2}+$ , $s_2 = s$.\\
Case 3: $\langle \tau_2 \pm_2 |\xi_2| \rangle$ dominant. \\
This case is reduced to
\begin{align*}
\Big| \int_* \frac{\widehat{u}_1(\tau_1,\xi_1)}{\langle |\tau_1| - |\xi_1|\rangle^{\frac{1}{2}+} \langle \xi_1 \rangle^{s+\frac{1}{2}}} 
\frac{\widehat{u}_2(\tau_2,\xi_2)}{\langle |\tau_2| - |\xi_2|\rangle^{0+} \langle \xi_2 \rangle^s}
\frac{\widehat{u}_3(\tau_3,\xi_3)}{\langle \xi_3 \rangle^{\frac{1}{4}-s}} \Big|
 \\
\lesssim \|u_1\|_{L^2_{xt}}\|u_2\|_{L^2_{xt}}\|u_3\|_{L^2_{xt}} \, ,
\end{align*}
which follows from Theorem \ref{Theorem3} with $s_0 = \frac{1}{4}-s$ , $b_0 = 0$ , $s_1=s+\frac{1}{2}$ , $b_1= \frac{1}{2}+,$  $b_2=0+$ , $s_2 = s$.

The cubic terms are easier to handle, because they contain one derivative less.\\
{\bf Claim 2:}
 \begin{align}
 \label{10}
&\|A_i \partial_j (|\phi|^2)\|_{X^{s-\frac{1}{4},0}_{\tau=0}}  \lesssim (\| \nabla A_i\|_{X^{s-\frac{1}{4},0}_{|\tau|=|\xi|}} +  \||\nabla|^{\epsilon}A_i\|_{L^{\infty}_t(L^2_x)}) \|\phi \|^2_{X^{s+1,\frac{1}{2}+}_{|\tau|=|\xi|}} \\
\label{11}
&\lesssim (\|\phi\|_{X^{s+1,\frac{1}{2}+}_{|\tau|=|\xi|}} + \|\partial_t \phi\|_{X^{s,\frac{1}{2}+}_{|\tau|=|\xi|}} \| \nabla A_i^{cf}\|_{X^{s-\frac{1}{4},0}_{\tau=0}} + \||\nabla|^{\epsilon}A_i^{cf}\|_{L^{\infty}_t(L^2_x)}) \|\phi \|^2_{X^{s+1,\frac{1}{2}+}_{|\tau|=|\xi|}} \, .
\end{align}
(\ref{10}) is proven by Sobolev's embedding theorem, especially $\dot{H}^{1-} \cap \dot{H}^{1+} \subset L^{\infty}$, and by splitting $A_i = A^l_i + A^h_i$ into low and high frequency parts, i.e., $supp \,\widehat{A}_i^l \subset \{ |\xi| \le 2 \}$ and $ supp \,\widehat{A}_i^h \subset \{|\xi| \ge 1\}$ . The low frequency part is easily taken care of as follows
\begin{align*}
\|A_i^l \partial_j (|\phi|^2)\|_{L^2_t H^{s-\frac{1}{4}}_x} 
  \lesssim \|A_j^l\|_{L^{\infty}_t H^{s,\infty}_x} \|\phi\|^2_{L^4_t H^{s+1}_x} \lesssim \||\nabla|^{\epsilon}A_j\|_{L^{\infty}_t L^2_x} \|\phi\|^2_{X^{s+1,\frac{1}{2}+}_{|\tau|=|\xi|}} \, .
 \end{align*}
For the high frequency part we obtain
\begin{align*}
&\|A_i^h \partial_j (|\phi|^2)\|_{L^2_t H^{s-\frac{1}{4}}_x} \\
& \lesssim \| \langle \nabla \rangle^{s-\frac{1}{4}} A_i^h\|_{L^2_t L^{\infty-}_x} \|\partial_j(|\phi|^2)\|_{L^{\infty}_t L^{2+}_x} + \|A_i^h\|_{L^2_t L^{\infty}_x} \| \langle \nabla \rangle^{s-\frac{1}{4}} \partial_j(|\phi|^2)\|_{L^{\infty}_t L^2_x} \\
& \lesssim \| \langle \nabla \rangle^{s-\frac{1}{4}} A_i^h\|_{L^2_t L^{\infty -}_x} \|\phi\|^2_{L^{\infty}_t H^{1+}_x} + \|A_i^h\|_{L^2_t H^{1+}_x} \||\phi|^2\|_{L^{\infty}_t H^{s+1}_x} \\
& \lesssim \| \nabla A_i^h\|_{X^{s-\frac{1}{4},0}_{|\tau|=|\xi|}}  \|\phi\|^2_{X^{s+1,\frac{1}{2}+}_{|\tau|=|\xi|}} \, .
\end{align*}
In order to obtain (\ref{11}) from these estimates it remains to estimate $A^{df}$.
We obtain by (\ref{***2'}) and Sobolev's embedding $\dot{H}^{1-\epsilon,\frac{2}{2-\epsilon}} \subset L^2$:
\begin{align}
\label{f}
\||\nabla|^{\epsilon}A_i^{df}\|_{L^{\infty}_t(L^2_x)} & \lesssim \| \phi \partial_t \phi \|_{L^{\infty}_t \dot{H}^{-1+\epsilon}_x} \lesssim \| \phi \partial_t \phi \|_{L^{\infty}_t L^{\frac{2}{2-\epsilon}}_x} \\
\nonumber
& \lesssim \|\phi\|_{L^{\infty}_t L^{\frac{4}{2-\epsilon}}_x} \| \partial_t \phi \|_{L^{\infty}_t L^{\frac{4}{2-\epsilon}}_x} \lesssim \|\phi\|_{X^{s+1,\frac{1}{2}+}_{|\tau|=|\xi|}} \|\partial_t \phi\|_{X^{s,\frac{1}{2}+}_{|\tau|=|\xi|}} \, .
\end{align}
Moreover for sufficiently small $\epsilon >0$ we obtain
\begin{equation}
\label{g}
\| \nabla A_i^{df}\|_{X^{s,\frac{1}{2}+\frac{\epsilon}{2}}_{|\tau|=|\xi|}} \lesssim \|\phi \partial_t \phi \|_{X^{s,\frac{1}{2}+\frac{\epsilon}{2}}_{|\tau|=|\xi|}} \lesssim \|\phi\|_{X^{s+1,\frac{1}{2}+\epsilon}_{|\tau|=|\xi|}} \|\partial_t \phi\|_{X^{s,\frac{1}{2}+\epsilon}_{|\tau|=|\xi|}}
\end{equation} 
by Theorem \ref{Theorem3} with $s_0= -s$ , $b_0 = -\frac{1}{2}-\frac{\epsilon}{2}$ , $s_1=s+1$ , $s_2 = s$ , $ b_1=b_2= \frac{1}{2}+\epsilon.$  This is more than we need here. \\
{\bf Claim 3:} 
 $$\|Im{(\overline{\phi} \partial_t \phi) |\phi|^2 \|_{X^{s,0}_{\tau =0}}     \lesssim \|\phi\|^3_{X^{s+1,\frac{1}{2}+}_{|\tau|=|\xi|}} \|\partial_t \phi\|_{X^{s,\frac{1}{2}+}_{|\tau|=|\xi|}}} \, . $$
This follows from
\begin{align*}
&\|Im(\overline{\phi} \partial_t \phi) |\phi|^2\|_{L^2_t H^{s}_x} \\
& \lesssim \|\phi^3\|_{L^{\infty}_t L^{\infty}_x} \|\partial_t \phi\|_{L^2_t H^s_x}  +  \|\phi^3\|_{L^{\infty}_t H^{s,\infty -}_x} \|\partial_t \phi\|_{L^2_t L^{2+}_x} \\
&\lesssim \|\phi\|^3_{L^{\infty}_t H^{s+1}_x}  \|\partial_t \phi\|_{L^2_t H^s_x} +  \|\phi\|^3_{L^{\infty}_t H^{s,\infty -}_x} \|\partial_t \phi\|_{L^2_t H^{0+}_x} \\ 
&\lesssim \|\phi\|^3_{X^{s+1,\frac{1}{2}+}_{|\tau|=|\xi|}}  \|\partial_t \phi\|_{X^{s,\frac{1}{2}+}_{|\tau|=|\xi|}}
\, .
\end{align*}
In order to control $\|A^{cf}\|_{L^{\infty}_t([0,T],L^2_x)}$ in the fixed point argument we have to estimate the $L^1_x([0,T],L^2_x)$ - norm of the right hand side in (\ref{***3a'}),(\ref{***3b'}). \\
{\bf Claim 4 (a)} 
$$ \int_0^T \|\nabla \phi \nabla \phi\|_{\dot{H}^{-1+\epsilon}} dt \lesssim \int_0^T \|\nabla \phi \nabla \phi\|_{L^\frac{2}{2-\epsilon}} dt \lesssim T \|\nabla \phi\|^2_{L^{\infty}_t L^\frac{4}{2-\epsilon}_x} \lesssim T \|\nabla \phi\|^2_{X^{s+1,\frac{1}{2}+}_{|\tau|=|\xi|}} $$
{\bf(b)} 
\begin{align*}
&\int_0^T \hspace{-0.2em}\| A^{df} \nabla(|\phi|^2)\|_{\dot{H}^{-1+\epsilon}_x} dt  \lesssim \int_0^T \hspace{-0.2em}\| A^{df} \nabla(|\phi|^2)\|_{L^{\frac{2}{2-\epsilon}}_x} dt \\
&\lesssim \int_0^T \hspace{-0.2em}\| \Delta^{-1} \nabla(\phi \partial_t \phi) \nabla(|\phi|^2)\|_{L^\frac{2}{2-\epsilon}_x} dt 
 \lesssim \int_0^T \| \Delta^{-1} \nabla(\phi \partial_t \phi)\|_{L^{\frac{4}{2-\epsilon}}_x} \|\nabla(|\phi|^2)\|_{L^{\frac{4}{2-\epsilon}}_x} dt \\
 &\lesssim \int_0^T \|\phi\|_{L^{\frac{8}{4-\epsilon}}_x} \|\partial_t \phi\|_{L^{\frac{8}{4-\epsilon}}_x} \|\phi\|_{L^{\infty}_x} \|\nabla \phi\|_{L^{\frac{4}{2-\epsilon}}_x} dt 
 \lesssim T \|\phi\|^3_{X^{s+1,\frac{1}{2}+}_{|\tau|=|\xi|}}  \|\partial_t \phi\|_{X^{s,\frac{1}{2}+}_{|\tau|=|\xi|}}
\end{align*}
{\bf (c)} 
\begin{align*}
& \int_0^T \| A^{cf} \nabla(|\phi|^2)\|_{\dot{H}^{-1+\epsilon}_x} dt \lesssim \int_0^T \|A^{cf} \nabla(|\phi|^2)\|_{L^{\frac{2}{2-\epsilon}}_x} dt \\
&\lesssim T \|A^{cf}\|_{L^{\infty}_t L^{\frac{2}{1-\epsilon}}_x} \|\phi\|_{L^{\infty}_t L^{\infty}_x} \|\nabla \phi\|_{L^{\infty}_t L^2_x} 
\lesssim T \|\phi\|^2_{X^{s+1,\frac{1}{2}+}_{|\tau|=|\xi|}} \||\nabla|^{\epsilon}A^{cf}\|_{L^{\infty}_t L^2_x}
\end{align*}
{\bf (d)} 
\begin{align*}
&\int_0^T \|\phi (\partial_t \phi) |\phi|^2\|_{\dot{H}^{-1+\epsilon}_x} dt \lesssim \int_0^T \|\phi (\partial_t \phi) |\phi|^2\|_{L^\frac{2}{2-\epsilon}} dt \\
& \lesssim T \|\phi\|^2_{L^{\infty}_x L^{\infty}_x} \|\phi\|_{L^{\infty}_t L^{\frac{4}{2-\epsilon}}_x} \|\partial_t \phi\|_{L^{\infty}_t L^{\frac{4}{2-\epsilon}}_x} \lesssim T \|\phi\|^3_{X^{s+1,\frac{1}{2}+}_{|\tau|=|\xi|}} \|\partial_t \phi\|_{X^{s,\frac{1}{2}+}_{|\tau|=|\xi|}} \, .
\end{align*}
Next in order to estimate $\|\phi\|_{X^{s+1,\frac {1}{2}+}_{|\tau|=|\xi|}}$ and $\|\partial_t \phi\|_{X^{s,\frac {1}{2}+}_{|\tau|=|\xi|}}$ we have to control the right hand side of (\ref{***1'}).\\
{\bf Claim 5:}
$$ \|A^{df} \nabla \phi\|_{X^{s,0}_{|\tau|=|\xi|}} \lesssim
\|\phi\|^2_{X^{s+1,\frac{1}{2}+}_{|\tau|=|\xi|}} \|\partial_t \phi\|_{X^{s,\frac{1}{2}+}_{|\tau|=|\xi|}}  \, . $$
We obtain by Sobolev:
\begin{align*}
\|A^{df} \nabla \phi \|_{X^{s,0}_{|\tau|=|\xi|}} &\lesssim \|A^{df}\|_{L^2_t H^{s,\infty-}_x} \|\nabla \phi \|_{L^{\infty}_t L^{2+}_x} + \|A^{df}\|_{L^2_t L^{\infty}_x} \|\nabla \phi \|_{L^{\infty}_t H^s_x} \\
& \lesssim (\||\nabla|^{\epsilon}A^{df}\|_{L^2_t L^2_x} + \|\nabla A^{df}\|_{L_t^2 H^s_x}) \|\nabla \phi\|_{L^{\infty}_t H^s_x} \, .
\end{align*}
Using (\ref{f}) and (\ref{g}) we obtain the claimed estimate. \\
{\bf Claim 6:}
$$ \|A^{cf} \nabla \phi\|_{X^{s,-\frac{1}{2}+}_{|\tau|=|\xi|}} \lesssim (\|\nabla A^{cf}_h\|_{X^{s-\frac{1}{4},0}_{\tau =0}} + \||\nabla|^{\epsilon} A^{cf}_l\|_{L^{\infty}_t L^2_x}) \|\nabla \phi\|_{X^{s,\frac{1}{2}+}_{|\tau|=|\xi|}} \, , $$
where we split $A^{cf}$ into its low and high frequency parts $A^{cf}_l$ and $A^{cf}_h$. The low frequency part is estimated as follows:
$$ \|A^{cf}_l \nabla \phi\|_{L^2_t H^s_x} \lesssim \|A^{cf}_l \|_{L^{\infty}_t H^{s,\infty}_x} \|\nabla \phi\|_{L^2_t H^s_x} \lesssim \||\nabla|^{\epsilon}A^{cf}\|_{L^{\infty}_t L^2_x} \|\nabla \phi\|_{X^{s,\frac{1}{2}+}_{|\tau|=|\xi|}} \, , $$
whereas for the high frequency part we use the trivial identity $X^{s-\frac{1}{4},0}_{\tau =0} = X^{s-\frac{1}{4},0}_{|\tau| =|\xi|}$ and
Theorem \ref{Theorem3} with $s_0=-s$ , $b_0=\frac{1}{2}-$ , $s_1=s+\frac{3}{4}$ , $b_1=0$ , $s_2=s$ , $b_2=\frac{1}{2}+,$  which requires $2s_0+s_1+b_1+2s_2= s+\frac{3}{4} > 1$, thus our assumption $s> \frac{1}{4}$ . \\
{\bf Claim 7:}
$$ \|\nabla A^{cf} \phi\|_{X^{s,-\frac{1}{2}+}_{|\tau|=|\xi|}} \lesssim \| \nabla A^{cf}\|_{X^{s-\frac{1}{4},\frac{1}{2}+}_{\tau =0}} \|\phi\|_{X^{s+1,\frac{1}{2}+}_{|\tau|=|\xi|}} \, . $$
By duality this is equivalent to
$$ \| w \phi \|_{X^{\frac{1}{4}-s,-\frac{1}{2}-}_{\tau =0}} \lesssim \|w\|_{X^{-s,\frac{1}{2}-}_{|\tau|=|\xi|}} \|\phi\|_{X^{s+1,\frac{1}{2}+}_{|\tau|=|\xi|}} \, . $$
We use the estimate $ \frac{\langle \xi \rangle}{\langle \tau \rangle}  \lesssim \langle |\xi| - |\tau| \rangle$ and obtain
$$ \| w \phi \|_{X^{\frac{1}{4}-s,-\frac{1}{2}-}_{\tau =0}} \lesssim \|w \phi\|_{X^{-s-\frac{1}{4} ++,\frac{1}{2}--}_{|\tau|=|\xi|}} \lesssim \|w\|_{X^{-s,\frac{1}{2}-}_{|\tau|=|\xi|}} \|\phi\|_{X^{s+1,\frac{1}{2}+}_{|\tau|=|\xi|}} \, , $$
where the last estimate follows from Theorem \ref{Theorem3} with $s_0 = s+\frac{1}{4} --$ , $b_0 = -\frac{1}{2}++,$ $s_1 = s+1$ , $s_2=-s $ , $b_1=\frac{1}{2}+$ , $b_2=\frac{1}{2}-$ . \\
{\bf Claim 8:} 
$$\|A^{cf} A^{cf} \phi \|_{X^{s,0}_{|\tau|=|\xi|}} \lesssim (\|\nabla A^{cf}\|^2_{X^{s-\frac{1}{2}-,\frac{1}{2}+}_{\tau =0}} + \||\nabla|^{\epsilon}A^{cf}\|^2_{L^{\infty}_t L^2_x}) \|\phi\|_{X^{s+1,\frac{1}{2}+}_{|\tau|=|\xi|}} \, . $$
Splitting $A^{cf}= A^{cf}_h + A^{cf}_l$ we first consider 
\begin{align*}
&\|A^{cf}_h A^{cf}_h \phi\|_{X^{s,0}_{|\tau|=|\xi|}} \lesssim
 \| \langle \nabla \rangle^s A^{cf}_h A^{cf}_h \phi\|_{L^2_t L^2_x} + \| A^{cf}_h A^{cf}_h \langle \nabla \rangle^s \phi\|_{L^2_t L^2_x}  \\
& \lesssim  \| \langle \nabla \rangle^s A^{cf}_h\|_{L^4_t L^{4-}_x} \| A^{cf}_h\|_{L^4_t L^{4+}_x} \| \phi\|_{L^{\infty}_t L^{\infty}_x} + \| A^{cf}_h \|^2_{L^4_t L^{4+}_x} \| \langle \nabla \rangle^s \phi\|_{L^{\infty}_t L^{\infty -}_x}  \\ 
& \lesssim  \| \langle \nabla \rangle^s A^{cf}_h\|_{L^4_t H^{\frac{1}{2}-}_x} \| A^{cf}_h\|_{L^4_t L^{4+}_x} \| \phi\|_{L^{\infty}_t L^{\infty}_x} + \| A^{cf}_h \|^2_{L^4_t H^{\frac{1}{2}+}_x} \| \langle \nabla \rangle^s \phi\|_{L^{\infty}_t L^{\infty -}_x}  \\ 
&  \lesssim \|\nabla A^{cf}_h\|_{X^{s-\frac{1}{2}-,\frac{1}{2}+}_{\tau =0}}^2  \|\phi\|_{X^{s+1,\frac{1}{2}+}_{|\tau|=|\xi|}} \, .
\end{align*}
Next we consider
$$
\|A^{cf}_l A^{cf}_l \phi\|_{X^{s,0}_{|\tau|=|\xi|}}  \lesssim \|A^{cf}_l\|_{L^{\infty}_t H^{s,\infty}_x}^2 \|\phi\|_{L^2_t H^s_x} 
\lesssim \||\nabla|^{\epsilon}A^{cf}\|_{L^{\infty}_t L^2_x}^2  \|\phi\|_{X^{s,\frac{1}{2}+}_{|\tau|=|\xi|}}  $$
and also
\begin{align*}
\|A^{cf}_l A^{cf}_h \phi\|_{X^{s,0}_{|\tau|=|\xi|}} & \lesssim \|A^{cf}_l\|_{L^{\infty}_t H^{s,\infty}_x} \|A^{cf}_h\|_{L^4_t H^{s,2+}_x} \|\phi\|_{L^4_t H^{s,\infty-}_x} \\ &\lesssim \||\nabla|^{\epsilon}A^{cf}_l\|_{L^{\infty}_t L^2_x} \| \nabla A^{cf}_h \|_{X^{s-\frac{1}{2}-,\frac{1}{2}+}_{|\tau|=|\xi|}} \|\phi\|_{X^{s+1,\frac{1}{2}+}_{|\tau|=|\xi|}} \, ,
\end{align*}
which completes the proof of claim 8.

If one combines similar estimates with (\ref{f}) and (\ref{g}) we also obtain the required bounds for $\|A^{df} A^{df} \phi\|_{X^{s,0}_{|\tau|=|\xi|}}$ and $\|A^{df} A^{cf} \phi\|_{X^{s,0}_{|\tau|=|\xi|}}$.\\
{\bf Claim 9:} For a suitable $N \in {\mathbb N}$ the following estimate holds:
$$\|\phi V'(|\phi|^2)\|_{X^{s,0}_{|\tau|=|\xi|}} \lesssim \|\phi\|_{X^{s+1,\frac{1}{2}+}_{|\tau|=|\xi|}} (1 + \|\phi\|_{X^{s+1,\frac{1}{2}+}_{|\tau|=|\xi|}}^N) \, .$$
Using the polynomial bounds of all derivatives of $V$ we crudely estimate using that $H^{s+1}$ is a Banach algebra:
$$ \|\phi V'(|\phi|^2)\|_{L^2_t H^s_x} \lesssim \|\phi\|_{L^2_t H^{s+1}_x}(1+\|\phi\|^N_{L^{\infty}_t H^{s+1}_x}) \lesssim \|\phi\|_{X^{s+1,\frac{1}{2}+}_{|\tau|=|\xi|}} (1 + \|\phi\|_{X^{s+1,\frac{1}{2}+}_{|\tau|=|\xi|}}^N) \, .$$

Now the contraction mapping principle applies. The claimed properties of $A^{df}$ follow immediately from (\ref{f}) and (\ref{g}), and the property $\nabla A^{cf} \in X^{s-\frac{1}{4},1-}_{\tau =0}[0,T]$ from claims 1-3. The proof of Theorem \ref{Theorem} is complete.
\end{proof}

\end{document}